\documentclass[11pt]{article}
\usepackage{amsthm,amsmath,amssymb}
\usepackage{natbib}
\usepackage{multirow}
\usepackage[pdftex]{graphicx}
\usepackage{caption}
\usepackage{subcaption}
\usepackage{float}
\usepackage{makecell}
\usepackage{booktabs}
\usepackage{array}
\usepackage{fullpage}
\usepackage{url}
\usepackage{algorithm}
\usepackage{tikz}
\usepackage{algorithmic}
\usepackage{bm}
\usepackage{cancel}
\usepackage{bbm}
\usepackage{smile}
\usepackage{mathtools}
\usepackage{lipsum}
\usepackage{mathrsfs}
\usepackage{dsfont}
\usepackage{titling}
\usepackage{epstopdf}
\usepackage{caption}
\usepackage{marvosym}
\usepackage{diagbox}
\usepackage{fullpage}
\usepackage{authblk}

\makeatletter
\DeclareRobustCommand\widecheck[1]{{\mathpalette\@widecheck{#1}}}
\def\@widecheck#1#2{%
    \setbox\z@\hbox{\m@th$#1#2$}%
    \setbox\tw@\hbox{\m@th$#1%
       \widehat{%
          \vrule\@width\z@\@height\ht\z@
          \vrule\@height\z@\@width\wd\z@}$}%
    \dp\tw@-\ht\z@
    \@tempdima\ht\z@ \advance\@tempdima2\ht\tw@ \divide\@tempdima\thr@@
    \setbox\tw@\hbox{%
       \raise\@tempdima\hbox{\scalebox{1}[-1]{\lower\@tempdima\box
\tw@}}}%
    {\ooalign{\box\tw@ \cr \box\z@}}}
\makeatother

\usepackage{natbib}
\usepackage{multirow}

\usepackage[colorlinks=true,
            linkcolor=blue,
            urlcolor=blue,
            citecolor=blue]{hyperref}

\usepackage[utf8]{inputenc}

\newcommand\independent{\protect\mathpalette{\protect\independenT}{\perp}}
\def\independenT#1#2{\mathrel{\rlap{$#1#2$}\mkern2mu{#1#2}}}

\newcommand*{\diam}{\operatorname{diam}}
\newcommand*{\TV}{\operatorname{TV}}

\newcommand*{\supp}{\mathrm{supp}}

\begin{document}

\title{\huge Nearly Minimax Optimal Wasserstein Conditional Independence Testing}
\author[1]{Matey Neykov}
\author[2]{Larry Wasserman\footnote{The last three authors are listed \href{https://tinyurl.com/44jnuv2u}{randomly}.}}
\newcommand\CoAuthorMark{\footnotemark[\arabic{footnote}]}
\author[3]{Ilmun Kim\protect\CoAuthorMark}
\author[4]{Sivaraman Balakrishnan\protect\CoAuthorMark}
\affil[1,2,4]{Department of Statistics \& Data Science, Carnegie Mellon University}
\affil[3]{Department of Statistics and Data Science, Yonsei University}

\maketitle{}

\abstract{This paper is concerned with minimax conditional independence testing. In contrast to some previous works on the topic, which use the total variation distance to separate the null from the alternative, here we use the Wasserstein distance. 
In addition, we impose Wasserstein smoothness conditions which on bounded domains are weaker than the corresponding total variation smoothness imposed, for instance, by \cite{neykov2021minimax}. This added flexibility expands the distributions which are allowed under the null and the alternative to include distributions which may contain point masses for instance. We characterize
the optimal rate of the critical radius of testing up to logarithmic factors. Our test statistic which nearly achieves the optimal critical radius is novel, and can be thought of as a weighted multi-resolution version of the $U$-statistic studied by \cite{neykov2021minimax}. 

\section{Introduction}

This paper focuses on conditional independence (CI) testing using the Wasserstein distance. CI testing is a fundamental problem in statistics. It has widespread applications in areas such as causal inference and causal discovery \citep{zhang2012kernel, spirtes2000causation, pearl2014probabilistic} and graphical models \citep{margaritis2005distribution,Koller2009Probabilistic}. In addition it is central to classical statistical concepts such as sufficiency or ancillarity \citep{dawid1979conditional}. On the other hand the Wasserstein distance, and its associated theory of optimal transport, which was originally introduced by \cite{monge1781memoire, kantorovich1942translocation}, has recently seen multiple applications in machine learning, and statistical methodology and theory: see for instance \citep{blanchet2019quantifying} for applications in robust machine learning,  \citep{rubner2000earth, sandler2011nonnegative, li2013novel} for applications in image analysis and  \citep{chernozhukov2017monge, hallin2021distribution, ghosal2022multivariate,manole2021plugin} which study optimal transport maps and use them to define multivariate analogues of the quantile of a distribution. Furthermore, also of note are recent uses of optimal transport in nonparametric hypothesis testing problems \citep{deb2021multivariate, deb2021efficiency}, distributional regression \citep{ghodrati2021distribution}, generative modeling \citep{finlay2020learning,onken2021ot}, fairness in machine learning \citep{gordaliza2019obtaining, black2020fliptest, de2021consistent} and statistical applications in the sciences \citep{komiske2020exploring}.

The Wasserstein distance is 
flexible and, unlike stronger metrics such as the total variation distance, can be small even when one compares continuous to discrete distributions. This versatility makes it 
attractive for problems in conditional independence testing where one may not want to assume a priori that the distribution does not contain point masses for example. It is in fact so natural to use the Wasserstein distribution in problems for CI that we are not the first to look into this problem. \cite{warren2021wasserstein} develops binning based tests for CI testing problems where the underlying conditional distributions are assumed to be Wasserstein smooth. On the surface, this is 
similar to what our paper is concerned with: under Wasserstein smoothness assumptions we formulate a binning based statistic. The main difference between our work and \cite{warren2021wasserstein} is our goal: we aim to find a (nearly) minimax optimal test statistic and characterize the minimax testing rate, whereas \cite{warren2021wasserstein} simply controls the type I and type II errors under certain sufficient conditions. This is a fundamental difference, and our test statistic is markedly distinct from the one used by \cite{warren2021wasserstein}: we use a weighted multiresolution $U$-statistic, whereas \cite{warren2021wasserstein} uses a plugin based statistic which compares the Wasserstein distributions on binned samples. This of course makes our analysis quite distinct from 
that of \citet{warren2021wasserstein}.

We will now give a high level overview of the minimax approach, inspired by \cite{ingster1982minimax,ingster2003nonparametric}, which we undertake. If a null distribution is very close in a certain metric (which in this paper we choose to depend on the Wasserstein distance), to an alternative distribution, tests will have difficulty in distinguishing whether a distribution is coming from the null or the alternative. To remedy this, one can remove distributions which are $\varepsilon_n$-close to the null hypothesis. Our goal is then to discover how small $\varepsilon_n$ can be (as a function of the sample size $n$), so that one can still distinguish the null from the alternative. In addition as we mentioned we impose smoothness conditions both under the null and under the alternative hypothesis. This additional requirement comes as no surprise, since \cite{shah2018hardness} proved that under no conditions CI testing is hard in the sense that the power under any alternative distribution of any test that controls the type I error over all (smooth and non-smooth) CI distributions below $\alpha$ is bounded by $\alpha$.

\subsection{Notation}

We now summarize commonly used notation throughout the paper. 

\begin{definition}[Total Variation Metric]\label{TV:def} The total variation (TV) metric between two distributions $p,q$ on a measurable space $(\Omega, \cF)$ is defined as
\begin{align*}
\TV(p,q) = \sup_{A \in \cF} |p(A) - q(A)| = \frac{1}{2}\|p-q\|_1 = \frac{1}{2}\int \bigg|\frac{dp}{d\nu} - \frac{dq}{d\nu}\bigg|d\nu,
\end{align*}
where the last identity assumes $\nu$ is a common dominating measure of $p$ and $q$, i.e., $p \ll \nu$, $q \ll \nu$ and $\frac{d p}{d \nu}, \frac{dq}{d\nu}$ denote the densities of $p$ and $q$ with respect to $\nu$ (note here that $\nu$ can always be taken as $\nu = p + q$).
\end{definition} 
We will now formalize our notation for conditional distributions. This notation is the same as the one used in \cite{neykov2021minimax} but for completeness we provide details here. If the triplet $(X,Y,Z)$ has a distribution $p_{X,Y,Z}$ we will use $p_{X,Y|Z = z}$ to denote the conditional joint distribution of $X,Y | Z = z$. Additionally $p_{X|Z = z}$ and $p_{Y|Z = z}$ will denote the marginal conditional distributions of $X|Z = z$ and $Y|Z = z$ respectively. The marginal distributions will be denoted with $p_X, p_Y, p_Z$ and joint marginal distributions will be denoted with $p_{X,Y}, p_{Y,Z}, p_{X,Z}$. Furthermore, with a slight abuse of notation, $p_{X,Y|Z}(x,y|z)$ and $p_{X | Z}(x|z)$ and $p_{Y| Z }(y|z)$ will denote the densities of these distributions evaluated at the points $x,y$ and $z$ (or the corresponding probability mass functions when $X$ and $Y$ are discrete).

In addition we will use $\lesssim$ and $\gtrsim$ to mean $\leq$ and $\geq$ up to positive universal constants (which may be different from place to place). If both $\lesssim$ and $\gtrsim$ hold we denote this as $\asymp$. For an integer $n \in \NN$ we use the convenient shorthand $[n] = \{1,2,\ldots,n\}$. 

Finally, for a real number $r \in \RR$ let $\lfloor r \rfloor$ be the largest integer smaller than or equal to $r$, and let $\lceil r\rceil$ be the smallest integer which is at least $r$. 

\subsection{Problem Formulation and Related Works}

In this section we formulate the problem precisely and mention some related works. Let $X,Y,Z \in [0,1]^3$ be three random variables. We are interested in testing $H_0: X \independent Y | Z$ versus the alternative $H_1: X \not\independent Y | Z$. Define the Wasserstein-$1$ distance
\begin{align*}
W_1(\mu,\nu) = \inf_{\gamma \in \Gamma(\mu, \nu)} \int \|x-y\|_2 d \gamma(x,y),
\end{align*}
where $\Gamma$ denotes the set of all couplings between $\mu$ and $\nu$ i.e., all joint distributions with marginals $\mu$ and $\nu$ and $\|\cdot\|_2$ denotes the Euclidean norm. Similarly one can define $W_2(\mu, \nu)$ as  
\begin{align*}
W_2(\mu,\nu) = \bigg[\inf_{\gamma \in \Gamma(\mu, \nu)} \int \|x-y\|^2_2 d \gamma(x,y)\bigg]^{1/2}.
\end{align*}

We will now state several well-known facts about the $W_1,W_2$ and $\TV$ distances which will be helpful throughout this work.
\begin{fact} \label{important:fact:about:W}The following statements hold true:
\begin{enumerate}
\item For any two probability distributions $p,q$ on $[0,1]^2$: 
\begin{align*}W_2(p,q) \lesssim \TV(p,q). \end{align*}
\item For any two probability distributions $p,q$ on $[0,1]^2$: 
\begin{align*}W_1(p,q) \leq W_2(p,q).\end{align*}
\item Wasserstein distance is a proper metric, i.e., for three distributions $p,q,r$ on $[0,1]^2$ we have 
\begin{align*}
W_i(p,q) \leq W_i (q,r) + W_i(r,p), ~~~~ i\in\{1,2\}.
\end{align*}
\item Squared Wasserstein-2 distance is sub-additive on product distributions, i.e., let $p_1,p_2,q_1,q_2$ be probability distributions on $[0,1]$, then 
\begin{align*}
W_2^2(p_1\times p_2,q_1\times q_2) \leq W_2^2(p_1,q_1)  + W_2^2(p_2,q_2).
\end{align*}
\item If $p,q$ are probability distributions on $[0,1]^2$ we have 
\begin{align*}
W_2^2(p,q) \leq \sqrt{2} W_1(p,q).
\end{align*}
\end{enumerate}
\end{fact}
We defer the proof of this result to the appendix. Let $\cP_0$ denote the set of all conditionally independent distributions supported on $[0,1]^3$, i.e. for all $q \in \cP_0$: $q_{X,Y|Z} = q_{X|Z}q_{Y|Z}$. 
\begin{assumption} \label{Wasserstein:smoothness:assumption}

Define the collection of probability distributions
\begin{align*}
\cP_0^{W}(L) := \{p \in \cP_0 : W_1(p_{X, Y |Z = z}, p_{X, Y |Z = z'}) \leq L |z - z'|, \mbox{ for all } z,z' \in [0,1]\}.
\end{align*}
Suppose that under the null hypothesis the distribution belongs to the class $\cP_0^{W}(L)$. 
\end{assumption}

In this paper we work exclusively with $W_1$ smoothness conditions as in the definition of $\cP_0^{W}(L)$ both under the null, and also under the alternative hypothesis as we will see shortly. Similar smoothness conditions have been used previously to enable binning based approaches to CI testing; see for instance \cite{neykov2021minimax, kim2021local} for total variation smoothness, and also \cite{warren2021wasserstein} for Wasserstein smoothness akin to the one we used above. One advantage of the $W_1$ smoothness in comparison with total variation smoothness as in \cite{neykov2021minimax}, is that on compact domains the $W_1$ distance is smaller than the total variation up to a constant \cite[See Fact \ref{important:fact:about:W} (1), and also Lemma 3 and Theorem 6.15 ][respectively]{slawski2022permuted, villani2009optimal}, and therefore, all previous examples suggested in Section 6 of \cite{neykov2021minimax}, which are total variation smooth also satisfy Wasserstein smoothness as defined in Assumption \ref{Wasserstein:smoothness:assumption}. Unlike total variation smoothness however, Wasserstein smoothness allows for distributions containing point masses; in other words being a mixture of discrete and continuous distributions may be Wasserstein smooth, while not being total variation smooth as is also pointed out by \cite{warren2021wasserstein}.

Let $\cP_1$ denote the class of all non-conditionally independent distributions i.e., the laws of all random variables $X,Y,Z \in [0,1]^3$ such that $X \not \independent Y | Z$. 
 
\begin{assumption}\label{assumption:W2:separation} 

Define the collection of alternative distributions $\cP_1^{W}(L,\varepsilon)$ as follows:
\begin{align*}
\cP_1^{W}(L,\varepsilon) := \{p \in \cP_1: & W_1(p_{X, Y |Z = z}, p_{X, Y |Z = z'}) \leq L |z - z'|, \mbox{ for all } z,z' \in [0,1],\\
& \inf_{q \in \cP_0} \EE_Z W_2(p_{X,Y|Z}, q_{X,Y|Z}) \geq \varepsilon\}.
\end{align*}
In the above definition, the expectation over $Z$ is taken with respect to the distribution $p_Z$ which is the $Z$-marginal of $p_{X,Y,Z}$. We will henceforth assume that the distributions under the alternative hypothesis belong to the class $\cP_1^{W}(L,\varepsilon) $. 

\end{assumption}
We would like to underscore that the $W_2$ distance is a popular distance which is often considered in practice. For instance \cite{rigollet2019uncoupled} use it to estimate the mean vector in the problem of uncoupled isotonic regression. 
As the reader can see we are using the $W_1$ distance to impose smoothness on the distributions while we are using the $W_2$ distance to impose separation between the null and the alternative. Using distinct measures of smoothness and separation is standard. 

See for instance \cite{ery2018remember} where the authors use H\"{o}lder smoothness on the densities and $L_2$ separation in goodness-of-fit problems. Furthermore, since the Wasserstein distance is monotonic (i.e., $W_1 \leq W_2$), the assumed smoothness in $W_1$ distance is weaker than the respective $W_2$ smoothness, hence in order to support more distributions we focus on the $W_1$ smoothness requirement. One final remark that we would like to make on Assumption \ref{assumption:W2:separation} is that for the same amount of separation --- $\varepsilon$ (up to universal constants) --- the Wasserstein separation discards more distributions as compared to the total variation distance. Formally we have
\begin{proposition}\label{W:sep:discards:more:distributions:than:TV:sep} If a distribution $p$ satisfies $\inf_{q \in \cP_0} \EE_Z W_2 (p_{X,Y|Z}, q_{X,Y|Z}) \geq \varepsilon$, then we also have $\inf_{q \in \cP_0} \TV(p,q) \gtrsim \varepsilon$.
\end{proposition}
\begin{proof}[Proof of Proposition \ref{W:sep:discards:more:distributions:than:TV:sep}]
To see this first observe that on bounded domains $W_2(p,q) \lesssim \TV(p,q)$ by Fact \ref{important:fact:about:W} 1, where we remind the reader that $\lesssim$ denotes inequality up to absolute constant factors. However, from Lemma B.4 of \cite{neykov2021minimax} we know 
\begin{align*}
\TV(p,q) \geq \EE_Z \TV(p_{X,Y|Z}, q_{X,Y|Z})/2 \gtrsim \EE_Z W_2(p_{X,Y|Z}, q_{X,Y|Z}).
\end{align*} 
Thus if two distributions $p$ and $q \in \cP_0$ satisfy $\EE_Z W_2 (p_{X,Y|Z}, q_{X,Y|Z}) \geq \varepsilon$, they also satisfy $\TV(p,q) \gtrsim \varepsilon$.
\end{proof}

To summarize, in comparison to \cite{neykov2021minimax}, the Wasserstein separation is ``stronger'' (by Proposition \ref{W:sep:discards:more:distributions:than:TV:sep}) than TV separation, while the Wasserstein smoothness requirement is ``weaker'' than the corresponding TV smoothness. In order to characterize the complexity of CI testing we use the minimax testing framework, introduced in the work of Ingster and co-authors \citep{ingster1982minimax,ingster2003nonparametric}, and which has since then been 
considered by many authors (see for instance \cite{lepski1999minimax,baraud2002non,diakonikolas2016new,valiant2017automatic,canonne2018testing,ery2018remember, canonne2020survey,balakrishnan2018hypothesis,balakrishnan2017hypothesis,carpentier2021optimal, neykov2021minimax,kim2022minimax,albert2022adaptive}).
Formally, consider the testing problem
\begin{align}
\label{eqn:main}
H_0: p_{} \in \cP_0^{W}(L) \mbox{ vs } H_1: p_{} \in \cP_1^W(L,\varepsilon).
\end{align}
We define the minimax risk of testing as 
\begin{align}\label{minimax:risk}
R_n(\varepsilon) = \inf_{\psi}\bigg\{\sup_{p \in \cP_0^{W}(L) } \EE_p[\psi(\cD_n)] + \sup_{p \in \cP_1^W(L,\varepsilon)} \EE_p [1 - \psi(\cD_n)]\bigg\}\footnotemark,
\end{align}
\footnotetext{Here with a slight abuse of notation, we use $\EE_p$ to denote expectation under i.i.d.~data $\cD_n  = \{(X_1,Y_1,Z_1), \ldots, (X_n, Y_n, Z_n)\}$ where each observation is drawn from $p$.}where the infimum is taken over all Borel measurable test functions $\psi: \supp(\cD_n) \mapsto [0,1]$ (which gives the probability of rejecting the null hypothesis), and $\supp(\cD_n)$ is the support of the random variables $\cD_n = \{(X_1,Y_1, Z_1),\allowbreak \ldots (X_n, Y_n, Z_n)\}$. In the development to follow, we assume that $L$ is a fixed non-zero constant which does not scale with $n$, and so we do not track the dependence of the critical radius on $L$.

In the minimax framework our goal is to study the \emph{critical radius} of testing defined as
\begin{align}\label{critical:radius}
\varepsilon_n(\cP_0^{W}(L), \cP_1^W(L,\varepsilon)) = \inf\bigg\{\varepsilon : R_n(\varepsilon) \leq \frac{1}{3}\bigg\}. 
\end{align}
The constant $\frac{1}{3}$ above is arbitrary, and can be chosen as any small constant. 
The minimax testing radius or the critical radius, corresponds to the smallest radius $\varepsilon$ at which there exists \emph{some test} which reliably distinguishes distributions in $\cH_0$ from those in $\cH_1$ which
are appropriately far from $\cH_0$. The critical radius provides a fundamental characterization of the statistical difficulty of the hypothesis testing problem in~\eqref{eqn:main}.

\subsection{Organization}

The remainder of the paper is structured as follows. In Section \ref{wasserstein:testing:section}, we formulate our test and prove it controls the type I and type II errors under an appropriate condition on the radius of separation. In Section \ref{lower:bound:section}, we state and prove our main lower bound. Finally, we conclude with a brief discussion of future work in Section \ref{discussion:section}.

\section{Wasserstein Testing}\label{wasserstein:testing:section}

In this section we present the main result of the paper. Our goal is to characterize the critical radius $\varepsilon _n$, defined in \eqref{critical:radius}. This involves upper and lower bounding it. Upper bounds are obtained by designing a test and analyzing its Type I and II errors (risk), and lower bounds are obtained via an information theoretic argument. The intuition behind our test construction is rooted in two propositions on the $W_2$ and $W_1$ distances given in the papers \cite{weed2019sharp, indyk2003fast} respectively. These Wasserstein distances can be thought of being approximately weighted ``multiresolution'' total variation distances (see Lemma~\ref{weed:bach:result} below).
 
Leveraging this result along with tests for distributions which are smooth in total variation \citep{neykov2021minimax}, we consider a multiresolution test statistic in order to approximates the $W_2$ separation functional. As we will see, the resulting test yields a nearly (up to logarithmic factors) minimax optimal Wasserstein CI test. The details on the upper bound are given below.

\subsection{Upper Bound}\label{upper:bound:section}

Construct $\cQ$, a collection of rectangular grids $Q^k$, $k \in \{1,\ldots, \lceil \log_2(d) \rceil\}$ with side Euclidean length (mostly) $\frac{1}{2^k}$ centered at a fixed point $\eta \in [0,1]^2$. Here $d$ is an integer defined as the number of bins used for the $Z$ variable. Each cell $A_{ij}^k \in Q^k$ is $A_{ij}^k = A_{i}^k\times A_{j}^{'k}$ where $A_i^k$ and $A_j^{'k}$  are intervals of size (mostly) $\frac{1}{2^k}$ on $[0,1]$ centered at the projections --- $\eta_1$ and $\eta_2$ --- of the point $\eta = (\eta_1,\eta_2)$ on the $x$ and the $y$ axis. We will now formally define the intervals $A_i^k$ for the convenience of the reader. Here the index $i$ ranges in the set $[L]$ where $L = 2 + \lfloor 2^k \eta_1 \rfloor  +  \lfloor 2^k (1-\eta_1) \rfloor $. We have 
\begin{align*}
A_1^k &= \bigg [0, \ \eta_1 - \frac{\lfloor 2^k \eta_1 \rfloor }{2^k}\bigg)\\
A_i^k & = \bigg [\eta_1 + \frac{i - 2 - \lfloor 2^k \eta_1 \rfloor}{2^k}, \ \eta_1 + \frac{i - 1 - \lfloor 2^k \eta_1 \rfloor}{2^k}\bigg),  \mbox{ for } i \in \{2, \ldots, L-1\}\\
A_L^k & = \bigg [\eta_1 + \frac{\lfloor 2^k (1-\eta_1) \rfloor }{2^k}, \ 1\bigg]
\end{align*}
Similarly, one can define the interval $A_j^{'k}$ for $j \in [L']$ where $L' = 2 + \lfloor 2^k \eta_2 \rfloor  +  \lfloor 2^k (1-\eta_2) \rfloor$. See also Figure \ref{fig:test} for a visualization of three such grids. We now restate and prove a proposition of \citep{weed2019sharp} adapted to our setting.

\begin{figure}
\centering

\begin{subfigure}{.3333\textwidth}
  \centering
\begin{tikzpicture}[scale=2.5]
    \def\gridsize{0.5}
    \def\centerx{0.3}
    \def\centery{0.4}
    
    \pgfmathsetmacro\halfgrid{\gridsize*4}
    
    \draw (0, 0) rectangle (1, 1);
    
    \foreach \x in { 0, 0.5}
        \draw (\centerx+\x, 0) -- (\centerx+\x, 1);
        
    \foreach \y in { 0, 0.5}
        \draw (0, \centery+\y) -- (1, \centery+\y);
    
    \filldraw (\centerx, \centery) circle (1pt) node[above left] {$\eta$};
        \filldraw (0, 0) node[below left] {$(0,0)$};
                \filldraw (1, 1) node[above right] {$(1,1)$};
\end{tikzpicture}
  \caption{$Q^1$ : grid of size $\frac{1}{2}$}
  \label{fig:sub1}
\end{subfigure}%
\begin{subfigure}{.333333\textwidth}
  \centering
\begin{tikzpicture}[scale=2.5]
    \def\gridsize{0.25}
    \def\centerx{0.3}
    \def\centery{0.4}
    
    \pgfmathsetmacro\halfgrid{\gridsize*4}
    
    \draw (0, 0) rectangle (1, 1);
    
    \foreach \x in {-0.25, 0,..., 0.5}
        \draw (\centerx+\x, 0) -- (\centerx+\x, 1);
        
    \foreach \y in {-0.25, 0,..., 0.5}
        \draw (0, \centery+\y) -- (1, \centery+\y);
    
    \filldraw (\centerx, \centery) circle (1pt) node[above left] {$\eta$};
        \filldraw (0, 0) node[below left] {$(0,0)$};
                \filldraw (1, 1) node[above right] {$(1,1)$};
\end{tikzpicture}

  \caption{$Q^2$: grid of size $\frac{1}{4}$}
  \label{fig:sub2}
\end{subfigure}

\begin{subfigure}{.333333\textwidth}
  \centering
\begin{tikzpicture}[scale=2.5]
    \def\gridsize{0.125}
    \def\centerx{0.3}
    \def\centery{0.4}
    
    \pgfmathsetmacro\halfgrid{\gridsize*4}
    
    \draw (0, 0) rectangle (1, 1);
    
    \foreach \x in {-.25, -0.125, 0,..., 0.5,.625}
        \draw (\centerx+\x, 0) -- (\centerx+\x, 1);
        
    \foreach \y in {-.375,-.25, -0.125, 0,..., 0.5}
        \draw (0, \centery+\y) -- (1, \centery+\y);
    
    \filldraw (\centerx, \centery) circle (1pt) node[above left] {$\eta$};
        \filldraw (0, 0) node[below left] {$(0,0)$};
                \filldraw (1, 1) node[above right] {$(1,1)$};
\end{tikzpicture}

  \caption{$Q^3$: grid of size $\frac{1}{8}$}
  \label{fig:sub3}
\end{subfigure}

\caption{Collection of grids centered at $\eta$}
\label{fig:test}

\end{figure}
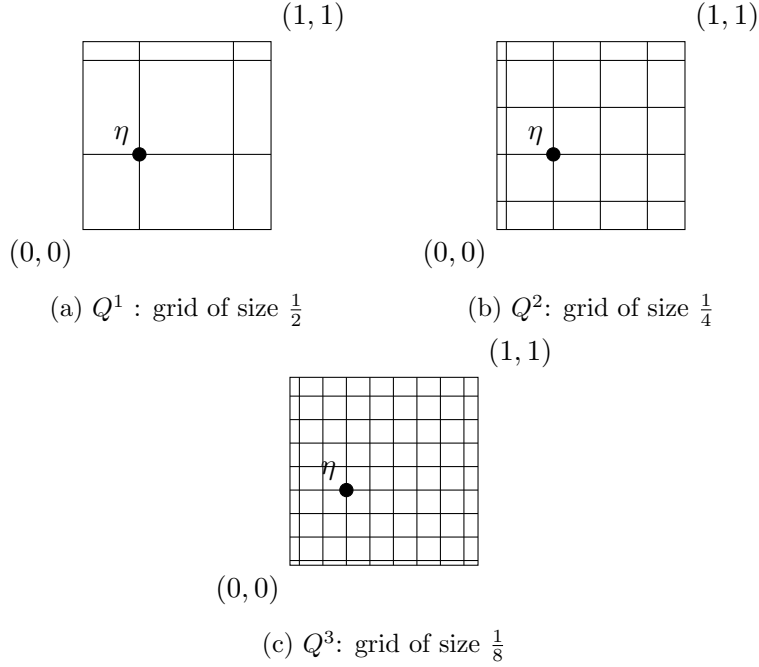

\begin{lemma}\label{weed:bach:result} For any two distributions $p$ and $q$ on $[0,1]^2$, we have the following inequality:
\begin{align}\label{Wasserstein:multiresolution:upper:bound}
W_2^2(p,q) \lesssim \frac{1}{2^{2\lceil \log_2(d) \rceil + 1} } + \sum_{ k = 1}^{\lceil \log_2(d) \rceil} \frac{1}{2^{2k}} \sum_{A^k_{ij} \in Q^k} |p(A^k_{ij}) - q(A^k_{ij})|.
\end{align}
\end{lemma}

Since the proof of Proposition \ref{weed:bach:result} follows directly from the result of \cite{weed2019sharp}, we defer it to the appendix. We now describe the test used for establishing an upper bound. First draw $N\sim Poi(n/2)$ samples. If $N > n$, accept the null hypothesis. If $N \leq n$, take the first $N$ samples out of the given $n$ samples and discard the rest. We bin the $Z$ support, i.e., $[0,1]$ in $d$ bins of equal size which we denote by $C_1, \ldots, C_d$. This separates the sample $\cD_N = \{(X_1, Y_1, Z_1), \ldots, (X_N,Y_N,Z_N)\}$ into smaller datasets $\cD_m = \{(X_i, Y_i) : Z_i \in C_m\}$. Let $|\cD_m| = \sigma_m$ denote the sample size of the $m$th bin. Define the function $g^k((x,y)) = (i,j)$ if and only if $(x,y) \in A^k_{ij} \in Q^k$. Next define the sets $\cD_m^k = \{g^k(X_i, Y_i) : Z_i \in C_m\}$ for $m \in [d]$ and $k \in 1,\ldots, \lceil \log_2(d) \rceil$. We now recall the definition of the $U$-statistic from \cite{neykov2021minimax}. For two observations $i$ and $j$ and two indices $x$ and $y$ consider the following expression
\begin{align*}
\phi_{ij}(xy) & = \mathbbm{1}(X_i = x, Y_i = y) - \mathbbm{1}(X_i = x) \mathbbm{1}(Y_j = y) .
\end{align*}
Note that $\phi$ takes a value among $\{-1,0,+1\}$. Next take four observations $i,j,k,l$ and consider the kernel
\begin{align*}
h((X_i, Y_i), (X_j, Y_j), (X_k, Y_k), (X_l, Y_l)) = \frac{1}{4!} \sum_{\pi \in [4!]} \sum_{x, y} \phi_{\pi_1\pi_2}(xy)\phi_{\pi_3 \pi_4}(xy),
\end{align*}
where $\pi$ is a permutation of $i,j,k,l$. Next, construct the corresponding $U$-statistic 
\begin{align*}
U_m(\cD_m^k) := \frac{1}{{\sigma_m \choose 4}} \sum_{i < j < k < l : (i,j,k,l) \in  \Sigma_m} h((X_i, Y_i), (X_j, Y_j), (X_k, Y_k), (X_l, Y_l)),
\end{align*}
where the summation is over choosing 4 distinct elements from $\Sigma_m$, where $\Sigma_m$ denotes the set of distinct indices in the set $\cD_m$. We now define the test statistic:
\begin{align*}
T := \EE_\eta \sum_{k = 1}^{\lceil \log_2(d) \rceil}   \frac{1}{2^{2k}}\sum_{m \in [d]}U(\cD_m^k) \mathbbm{1}(\sigma_m \geq 4) \sigma_m,
\end{align*}
where the expectation above over $\eta$ is taken with respect to uniformly sampling $\eta$ on the grid points of a square grid of side Euclidean length equal to $1/(2^{\lceil \log_2(d)\rceil + 1})$ on $[0,1]^2$ centered at $\mathbf{0} = (0,0)$. We then define the test
\begin{align}\label{test:def}
\psi_\tau(\cD_N) = \mathbbm{1}(T \geq \tau)  \mathbbm{1}(N \leq n).
\end{align}

\begin{remark}[On computing the test $\psi_\tau(\cD_N)$] According to a careful analysis in Section 3.3 of \cite{kim2023conditional} calculating $U(\cD_m^k)$ can be done in $O(\sigma_m)$ operations. This implies that (for a fixed $\eta$) calculating $\sum_{m \in [d]}U(\cD_m^k) \mathbbm{1}(\sigma_m \geq 4) \sigma_m$ takes at most $O(n)$ time; then $\sum_{k = 1}^{\lceil \log_2(d) \rceil}   \frac{1}{2^{2k}}\sum_{m \in [d]}U(\cD_m^k) \mathbbm{1}(\sigma_m \geq 4) \sigma_m$ takes $O(\log_2(d) n)$ time. Since in the end we set $d \asymp n^{2/5}$ for a fixed $\eta$ we have $O(n \log_2(n))$ operations. Finally, since $\eta$ belongs to a grid of at most $16 d^2$ points, we have that the computational cost is $O(d^2 \log_2(d) n) = O(n^{9/5}\log_2(n))$. This is bigger than linear complexity so it can be prohibitive for a large $n$. However we note that the computational complexity of calculating $T$ is better than quadratic time.
\end{remark}

We are now ready to state the main result of the paper.

\begin{theorem}\label{main:theorem} Take $d = \lceil n^{2/5}\rceil$ and set $\tau = \zeta \sqrt{d} \log_2^2 d$ for a sufficiently large constant $\zeta$. Suppose that $\varepsilon > \frac{c (\log_2 d)^{3/4}}{d^{1/2}}$, for a sufficiently large constant $c$. Then 
\begin{align*}
\sup_{p \in \cP_0^{W}(L)} \EE_p \psi_\tau(\cD_N) & \leq \frac{1}{10},\\
\sup_{p \in \cP_1^{W}(L,\varepsilon)} \EE_p (1 - \psi_\tau(\cD_N)) & \leq \frac{1}{10} + \exp(-n/8).
\end{align*}
\end{theorem}
Setting $d = \lceil n^{2/5}\rceil$, the above result establishes an upper bound for the critical radius as
\begin{align*}
	\varepsilon_n(\cP_0^{W}(L), \cP_1^W(L,\varepsilon)) \leq c_1 \frac{(\log_2 n)^{3/4}}{n^{1/5}},
\end{align*}
where $c_1$ is some positive constant. We now compare this rate to the rates given in \cite{neykov2021minimax}. There are two main results in \cite{neykov2021minimax} regarding the separation radius. 
\begin{enumerate}
\item First we comment on the ``fully'' continuous setting. In this setting, \cite{neykov2021minimax} assume that the distributions are TV smooth, i.e. that the conditional distributions are Lipschitz in the TV sense as a function of the conditioning variable $Z$. Additionally, they 
assume that the distributions $X,Y|Z= z$ have H\"{o}lder continuous density functions with exponent $s$ for all $z$ (see Definitions 2.3 and 2.4 in \cite{neykov2021minimax} for more details). The critical radius in their work scales as $n^{-2s/(5s + 2)}$, which is faster than the $n^{-1/5}$ rate we obtain in this paper, for all sufficiently large $s$ values. 
In this paper we assume that the conditional distributions are Lipschitz in the $W_1$ sense, but in stark contrast to the work of \cite{neykov2021minimax} we do not require additional smoothness on the distributions (such as H\"{o}lder smoothness). While this results in a slower rate, we earn flexibility in terms of the allowed distributions. Indeed, this flexibility is the main benefit afforded by testing using the Wasserstein distance. When testing under separation in the TV metric, even problems simpler than CI testing, such as goodness-of-fit testing, are impossible without additional smoothness assumptions \citep{balakrishnan2017hypothesis}. This is however not the case for testing with separation in the Wasserstein distance \citep{ba2011sublinear}, which is a tractable task even without smoothness assumptions.
\item The Wasserstein smoothness assumptions we impose can also support discrete distributions, and hence it is also sensible to compare our rates with the discrete case considered by \cite{neykov2021minimax}. The rate in the TV smoothness setting is $n^{-2/5}$, which is faster than the $n^{-1/5}$ that we established above. We can conclude that even with stronger separation requirement we impose, the problem of Wasserstein testing is harder than TV testing in the discrete case considered by \cite{neykov2021minimax}.
\end{enumerate}

\noindent The remaining of this section is devoted to the proof of Theorem \ref{main:theorem}. 

\subsection{Proof of Theorem \ref{main:theorem}}
Similarly to the proof of Theorem 5.2 of \cite{neykov2021minimax}, it suffices to show the result assuming that $N \sim Poi(n)$. We will analyze the expectation and variance of $T$ in Section~\ref{Section: Analysis of the Expectation} and Section~\ref{Section: Analysis of the Variance}, respectively. We do so in order to apply Chebyshev's inequality and control the risk from above (see Section~\ref{chebyshevs:ineq:section}). 

\subsubsection{Analysis of the Expectation} \label{Section: Analysis of the Expectation}
In this section we are concerned with controlling the expectation $\EE T$ from below and above under the alternative and the null hypothesis respectively.   Fix $k \in \{1,\ldots, \lceil \log_2(d) \rceil\}$. Starting with the expectation, conditional on $\sigma_m$ with $\sigma_m \geq 4$, we have that 
$$
\EE[U(\cD_m^k) | \sigma_m] = \sum_{i,j} (q^k_{ij}(m) - q_{i\cdot}^k(m)q_{\cdot j}^k(m))^2,
$$
where $q^k_{ij}(m) = P_{X,Y|Z \in C_m}(A_{ij}^k | Z \in C_m)$ and $q_{i \cdot}^k(m) = \sum_j q^k_{ij}(m) = P_{X|Z \in C_m}(A^k_i | Z \in C_m)$, and similarly for $q_{\cdot j}^k(m)$. With a slight abuse of notation, we define the expression $U(\cD_m^k) := \EE[U(\cD_m^k) | \sigma_m] :=  \sum_{i,j} (q^k_{ij}(m) - q_{i\cdot}^k(m)q_{\cdot j}^k(m))^2$ even when $\sigma_m < 4$ even though in this case the $U$-statistic $U(\cD_m^k)$ is not well defined. This is a legitimate operation, since our test statistic $T$ does not ``see'' the values of the $U$-statistic for $m$ such that $\sigma_m < 4$. In other words, since the indicator $\mathbbm{1}(\sigma_m \geq 4) U(\cD_m^k) = 0$ when $\sigma_m < 4$ we can define the value of $U(\cD_m^k)$ to be $\sum_{i,j} (q^k_{ij}(m) - q_{i\cdot}^k(m)q_{\cdot j}^k(m))^2$. Let $p_m = \PP(Z \in C_m)$. 

\paragraph{Analysis under the Alternative Hypothesis.} The goal of this section is to lower bound $\EE T$ under the alternative. We start by looking into the following expression
\begin{align}
\sum_{m\in [d]} \sqrt{\EE[ \EE_\eta [U(\cD_m^k)] | \sigma_m]} p_m & = \sum_{m\in [d]} \sqrt{\EE_\eta \sum_{i,j} \bigl(q^k_{ij}(m) - q_{i\cdot}^k(m)q_{\cdot j}^k(m)\bigr)^2 \nonumber
} p_m \nonumber\\
&  \geq \EE_\eta \sum_{m\in [d]} \sqrt{ \sum_{i,j} \bigl(q^k_{ij}(m) - q_{i\cdot}^k(m)q_{\cdot j}^k(m)\bigr)^2
} p_m \nonumber\\
& \geq  \sum_{m\in [d]} \frac{\sum_{i,j}  \EE_\eta|q^k_{ij}(m) - q_{i\cdot}^k(m)q_{\cdot j}^k(m)|}{2^k + 1} p_m\nonumber\\
& \geq \sum_{m\in [d]} \frac{\sum_{i,j} \EE_\eta |q^k_{ij}(m) - q_{i\cdot}^k(m)q_{\cdot j}^k(m)|}{2^{k + 1}} p_m\label{bound:sqrt:EUD}
\end{align}
where we used Jensen's inequality, the fact that $\sqrt{\sum_{i = 1}^l a_i^2} \geq \sum_{i = 1}^l a_i/\sqrt{l}$ for any real numbers $a_i \in \RR$, and the fact that there are at most $(2^k + 1)^2$ cells in $Q^k$ (here observe that $L \leq 2^k + 2$ (as defined in the beginning of Section \ref{upper:bound:section}); however, $L$ can be $2^k + 2$ only when $\lfloor 2^k \eta_1\rfloor$ and $\lfloor 2^k (1-\eta_1)\rfloor$ are both integers in which case $A_1^k$, $A_L^k$ are $\varnothing$ so that we effectively have $2^k$ intervals in that case; hence we have at most $(2^k + 1)^2$ cells in $Q^k$ since the same logic is valid for $L'$). We will now need the following result which quantifies the error in approximation of the expected $W_2$ with its binned counterpart. 

\begin{lemma}\label{Wasserstein:smoothness:lemma} If the distribution $p_{X,Y,Z}$ is Wasserstein $1$-smooth, i.e., $W_1(p_{X,Y|Z = z}, p_{X,Y|Z=z'}) \leq L|z - z'|$ we have that 
\begin{align*}
\MoveEqLeft \varepsilon \leq \inf_{q \in \cP_0} \EE_Z W_2(p_{X,Y|Z}, q_{X,Y|Z})  \leq \int W_2(p_{X,Y|Z = z}, p_{X| Z = z} p_{Y | Z = z}) d P(z) \\
& \leq \sum_{m \in [d]} W_2(p_{X,Y|Z \in C_m}, p_{X| Z \in C_m} p_{Y | Z \in C_m}) p_m + \kappa (L \max_{m \in [d]} \diam(C_m))^{1/2},
\end{align*}
where $\kappa$ is an absolute constant.
\end{lemma}
\begin{remark}
By the elementary inequality $(a + b)^2 \leq 2a^2 + 2b^2$, and the convexity of $x\mapsto x^2$ we have
\begin{align}\label{W2pm:bound}
 \varepsilon^2 & \leq  \Bigl(\inf_{q \in \cP_0} \EE_Z W_2(p_{X,Y|Z}, q_{X,Y|Z}) \Bigr)^2 \nonumber \\
& \leq  \Bigl(\sum_{m \in [d]} W_2(p_{X,Y|Z \in C_m}, p_{X| Z \in C_m} p_{Y | Z \in C_m}) p_m + \kappa L^{1/2} \max_{m \in [d]} \diam(C_m)^{1/2}\Bigr)^2 \nonumber\\
& \lesssim  \sum_{m \in [d]} W^2_2(p_{X,Y|Z \in C_m}, p_{X| Z \in C_m} p_{Y | Z \in C_m}) p_m  + \kappa^2 L \max_{m \in [d]} \diam(C_m).
\end{align}
\end{remark}
\noindent We defer the proof of Lemma \ref{Wasserstein:smoothness:lemma} to the Appendix. Continuing the bound \eqref{bound:sqrt:EUD} we conclude that 
\begin{align} \nonumber
\sum_{k = 1}^{\log_2(d)}\sum_{m\in [d]} \frac{1}{2^{k}}\sqrt{\EE[\EE_\eta[U(\cD_m^k)] | \sigma_m]} p_m &\geq \sum_{k} \sum_{m\in [d]} \frac{\sum_{i,j} \EE_\eta |q^k_{ij}(m) - q_{i\cdot}^k(m)q_{\cdot j}^k(m)|}{2^{2k + 1}} p_m \\ \nonumber
&\geq C \sum_{m \in [d]} W_2^2(p_{X,Y|Z \in C_m}, p_{X|Z \in C_m}p_{Y|Z \in C_m}) p_m - \frac{1}{d^2}\\ 
& \geq C \varepsilon^2 -  \frac{C' L}{d} -  \frac{1}{d^2} =: \Upsilon, \label{Eq: Upsilon}
\end{align}
where $C$ is some absolute constant from \eqref{Wasserstein:multiresolution:upper:bound}, and the $-\frac{1}{d^2}$ comes from the term $\frac{1}{2^{2\lceil \log_2(d)\rceil}}$ 
where the term $\frac{C'L}{d}$ comes from Lemma \ref{Wasserstein:smoothness:lemma}, and more specifically from the last term on the right hand side of \eqref{W2pm:bound}. Note also that the inequality of Lemma \ref{weed:bach:result} holds for any $\eta$, which means that it also holds in expectation.  

Next by Lemma 3.1 of \cite{canonne2018testing}
we have
\begin{align*}
\sum_k \frac{1}{2^{2k}} \sum_{m \in [d]} \EE[\EE_\eta[U(\cD_m^{k})] | \sigma_m] \EE[\sigma_m \mathbbm{1}(\sigma_m \geq 4)] & \geq \gamma \sum_k \frac{1}{2^{2k}}  \sum_{m: (n p_m) > 1} \EE[\EE_\eta[U(\cD_m^{k})]| \sigma_m] n p_m \\
&+ \gamma \sum_k \frac{1}{2^{2k}}   \sum_{m: (n p_m) \leq 1} \EE[\EE_\eta[U(\cD_m^{k})]| \sigma_m] (n p_m)^4,
\end{align*}
for an absolute constant $\gamma$. Since by \eqref{Eq: Upsilon} we have that $\sum_{k = 1}^{\log_2(d)}\sum_{m\in [d]} \frac{1}{2^{k}}\sqrt{\EE[\EE_\eta[U(\cD_m^k)] | \sigma_m]} p_m  \geq \Upsilon$ we have that either 
\begin{align}
\sum_{k = 1}^{\log_2(d)} \frac{1}{2^{k}} \sum_{m: (np_m) > 1} \sqrt{\EE[\EE_\eta[U(\cD_m^{k})]|\sigma_m]} n  p_m & \geq \frac{n \Upsilon}{2}, \mbox{ or } \label{first:case}\\
\sum_{k = 1}^{\log_2(d)} \frac{1}{2^{k}} \sum_{m: (np_m) \leq 1} \sqrt{\EE[\EE_\eta[U(\cD_m^{k})]|\sigma_m]} n  p_m & \geq \frac{n \Upsilon}{2}.\label{second:case}
\end{align} 
We now consider two cases:
\begin{itemize}
\item[i.] In the first case we assume \eqref{first:case} (where we remind the reader that $\Upsilon$ is defined in \eqref{Eq: Upsilon}). By the Cauchy--Schwarz inequality, we have 
\begin{align*}
\sum_{m: (n p_m) > 1} \EE[\EE_\eta[U(\cD_m^{k})] | \sigma_m] n p_m & \geq \frac{(\sum_{m: (np_m) > 1} \sqrt{\EE[\EE_\eta[U(\cD_m^{k})]|\sigma_m]} n  p_m)^2}{\sum_{m: (n p_m) > 1} n p_m} \\
& \geq \frac{(\sum_{m: (np_m) > 1} \sqrt{\EE[\EE_\eta[U(\cD_m^{k})]|\sigma_m]} n  p_m)^2}{n} . 
\end{align*}
Hence
\begin{align*}
\sum_{k} \frac{1}{2^{2k}} \sum_{m: (n p_m) > 1} \EE[\EE_\eta[U(\cD_m^{k})] | \sigma_m] n p_m &\geq \sum_{k} \frac{(\sum_{m: (np_m) > 1} \frac{1}{2^{k}}\sqrt{\EE[\EE_\eta[U(\cD_m^{k})]|\sigma_m]} n  p_m)^2}{n} \\
&\geq \frac{(\sum_{k}\sum_{m: (np_m) > 1} \frac{1}{2^{k}}\sqrt{\EE[\EE_\eta[U(\cD_m^{k})]|\sigma_m]} n  p_m)^2}{n \lceil \log_2 d\rceil}\\
& \gtrsim \frac{n \Upsilon^2}{\lceil \log_2 d\rceil}.
\end{align*}

\item[ii.] In the second case we suppose \eqref{second:case} holds.
Note that for any non-negative sequences $\{a_m\}_{m=1}^n$ and $\{b_m\}_{m=1}^n$, Jensen's inequality yields
\begin{align*}
	\sum_{m=1}^n \frac{a_m^{1/3}}{\sum_{j=1}^n a_j^{1/3}} a_m^{2/3} b_m^4 \geq \biggl( \sum_{m=1}^n \frac{a_m^{1/3}}{\sum_{j=1}^n a_j^{1/3}} a_m^{1/6} b_m \biggr)^4.
\end{align*}
Taking 
\begin{align*}
	a_m = \sum_k \frac{1}{2^{2k}}\EE[\EE_\eta[U(\cD_m^{k})] | \sigma_m] \quad \text{and} \quad b_m = np_m,
\end{align*}
we have
\begin{align*}
\MoveEqLeft \bigg(\sum_{m : (np_m) \leq 1}  \Bigl(\sum_k \frac{1}{2^{2k}}\EE[\EE_\eta[U(\cD_m^{k})] | \sigma_m]\Bigr)^{1/3}\bigg)^3 \sum_k \frac{1}{2^{2k}} \sum_{m : (np_m) \leq 1} \EE[\EE_\eta[U(\cD_m^{k})] | \sigma_m] (n p_m)^4 \\
& \geq \bigg(\sum_{m : (np_m) \leq 1} \sqrt{\sum_k \frac{1}{2^{2k}}\EE[\EE_\eta[U(\cD_m^{k})] | \sigma_m]} n p_m\bigg)^4,
\end{align*}
and therefore
\begin{align*}
\sum_k \frac{1}{2^{2k}}\sum_{m : (np_m) \leq 1} \EE[\EE_\eta[U(\cD_m^{k})] | \sigma_m] (n p_m)^4 \gtrsim \bigg(\sum_{m : (np_m) \leq 1} \sqrt{\sum_k \frac{1}{2^{2k}}\EE[\EE_\eta[U(\cD_m^{k})] | \sigma_m]} n p_m\bigg)^4/d^3,
\end{align*}
since $\EE[\EE_\eta[U(\cD_m^{k})]| \sigma_m] \leq \EE_\eta\big(\sum_{x,y} |q_{xy}(m) -q_{x\cdot}(m)q_{\cdot y}(m)|\big)^2 \leq 4$, and the summation over $k$ reduces to a converging geometric series and finally $|\{m : (n p_m) \leq 1\}| \leq d$. Now, by the Cauchy--Schwarz inequality,
\begin{align*}
\sqrt{\sum_k \frac{1}{2^{2k}}\EE[\EE_\eta[U(\cD_m^{k})] | \sigma_m]} \geq \sum_k \frac{1}{2^k} \sqrt{\EE[\EE_\eta[U(\cD_m^{k})] | \sigma_m])}/\lceil \log_2 d\rceil^{1/2}
\end{align*}
so we conclude
\begin{align*}
\sum_k \frac{1}{2^{2k}}\sum_{m : (np_m) \leq 1} \EE[\EE_\eta[U(\cD_m^{k})] | \sigma_m] (n p_m)^4 \gtrsim \frac{(n \Upsilon)^4}{d^3 \lceil \log_2 d\rceil^2}.
\end{align*}

\end{itemize}
Combining the above results, we have established that under the alternative,
\begin{align} \label{Eq: lower bound on expectation}
	\EE[T] \gtrsim \min \biggl\{ \frac{n \Upsilon^2}{\lceil \log_2 d\rceil}, \, \frac{(n \Upsilon)^4}{d^3 \lceil \log_2 d\rceil^2} \bigg\}.
\end{align}

\paragraph{Analysis under the Null Hypothesis.} Next we will upper bound the expectation of $T$ under the null hypothesis:
\begin{align*}
\sum_{m \in [d]} \EE[\EE_\eta[U(\cD_m^{k})] | \sigma_m] \EE[\sigma_m \mathbbm{1}(\sigma_m \geq 4)] & \leq  n\sum_{m \in [d]} \EE[\EE_\eta[U(\cD_m^{k})] | \sigma_m] p_m. \\
& = n \sum_{m \in [d]} \EE_\eta \sum_{i,j} (q^k_{ij}(m) - q_{i\cdot}^k(m)q_{\cdot j}^k(m))^2 p_m\\
& \leq n \sum_{m \in [d]}\EE_\eta (\sum_{i,j} |q^k_{ij}(m) - q_{i\cdot}^k(m)q_{\cdot j}^k(m)|)^2 p_m\\
& \lesssim n \sum_{m \in [d]}\EE_\eta (\sum_{i,j} |q^k_{ij}(m) - q_{i\cdot}^k(m)q_{\cdot j}^k(m)|) p_m,
\end{align*}
where we remind the reader that we assume $\EE[U(\cD_m^{k}) | \sigma_m] = (q^k_{ij}(m) - q_{i\cdot}^k(m)q_{\cdot j}^k(m))^2$ for all $m$ (even though the value of $U(\cD_m^{k})$ is technically only defined for $m : \sigma_m \geq 4$).
We now remind the reader that $q^k_{ij}(m) = P_{X,Y|Z \in C_m}(A_{ij}^k | Z \in C_m)$ and $q_{i \cdot}^k(m) = \sum_j q^k_{ij}(m) = P_{X|Z \in C_m}(A^k_i | Z \in C_m)$, and similarly for $q_{\cdot j}^k(m)$. Next we will handle the expression 
\begin{align*}
\MoveEqLeft \sum_{i,j} |q^{k}_{ij}(m) - q_{i\cdot}^{k}(m)q_{\cdot j}^{k}(m)| \\
& = \sum_{i,j} \bigg |\int_{C_m} P(A_{ij}^k | Z = z) d\tilde P(z) - \int_{C_m} P(A_{i}^k | Z = z) d\tilde P(z) \int_{C_m} P(A_{j}^{'k} | Z = z) d\tilde P(z)\bigg|\\
& = \sum_{i,j} \bigg |\int_{C_m} P(A_{i}^k | Z = z)P(A_{j}^{'k} | Z = z) d\tilde P(z) - \int_{C_m} P(A_{i}^k | Z = z) d\tilde P(z) \int_{C_m} P(A_{j}^{'k} | Z = z) d\tilde P(z)\bigg|\
  \\  &\leq  \int_{C_m} \sum_{i}\bigg|P_{X | Z = z}(A^k_i) - \int_{C_m} P_{X | Z = z}(A^k_i) d \tilde P(z)\bigg|\\
& \times \sum_{j}\bigg| P_{Y | Z = z}(A_j^{'k}) - \int_{C_m} P_{X | Z = z}(A_j^{'k}) d \tilde P(z)\bigg| d \tilde P(z),
\end{align*}
by Jensen's inequality and where $\tilde P(z) = d P(z)/P(Z \in C_m)$. 

We will now argue that the above is smaller than or equal to the product of total variations. Take the first term. By Jensen's inequality
\begin{align*}
& \sum_{i}\bigg|P_{X | Z = z}(A^k_i) - \int_{C_m} P_{X | Z = z}(A^k_i) d \tilde P(z)\bigg| \leq \int_{C_m}\sum_{i}\bigg|P_{X | Z = z}(A^k_i) -  P_{X | Z = z'}(A^k_i) \bigg|d \tilde P(z')\\
&=2\int_{C_m} d_{\operatorname{TV}}(P^k_{X | Z = z} ,  P^k_{X | Z = z'})d \tilde P(z'),
\end{align*}
where $P^k$ denotes the discretized distributions on the grid.
We now have

\begin{align*}
\MoveEqLeft \EE_\eta \sum_k \frac{1}{2^{2k}} n \sum_{m \in [d]} \biggl(\sum_{i,j} |q^k_{ij}(m) - q_{i\cdot}^k(m)q_{\cdot j}^k(m)|\biggr) p_m \\
& \leq \sum_{m \in [d]}\EE_\eta \sum_k \frac{1}{2^{2k}} n \int_{C_m} \int_{C_m} \int_{C_m}  4 d_{\operatorname{TV}}(P^k_{X|z}, P^k_{X|z'})d_{\operatorname{TV}}(P^k_{Y|z}, P^k_{Y|z''})d \tilde P(z')d \tilde P(z'') d \tilde P(z) p_m\\
&\leq 4 n \sum_{m \in [d]} \int_{C_m} \int_{C_m} \int_{C_m} \biggl( \EE_{\eta_1}\sum_k \frac{1}{2^k} d_{\operatorname{TV}}(P^k_{X|z}, P^k_{X|z'})\biggr)\times 
\\&~~~~~~\qquad\qquad\qquad\qquad~~ \biggl( \EE_{\eta_2}\sum_k \frac{1}{2^k} d_{\operatorname{TV}}(P^k_{Y|z}, P^k_{Y|z''})\biggr)d \tilde P(z')d \tilde P(z'') d \tilde P(z) p_m\\
& \lesssim n (\log_2 d)^2/d^2,
\end{align*}
since by Lemma \ref{indyk:thaper:on:a:grid}, proved below, the summations are bounded as:
\begin{align*}
2 \EE_{\eta_1} \sum_k 1/2^k d_{\operatorname{TV}}(P^k_{X|z}, P^k_{X|z'}) & \leq (\lceil \log_2(d)\rceil + 1)4 \bigg(W_1(P_{X|z}, P_{X|z'}) +\frac{1}{2^{\lceil \log_2(d)\rceil + 1}}\bigg)\\
& \lesssim \log_2 d (L/d + 1/d),
\end{align*}
using the Wasserstein smoothness as in Assumption \ref{Wasserstein:smoothness:assumption} and also inequality \eqref{W_1:joint:bigger:than:marginals}.
Hence under the null, we have 
\begin{align*}
\EE [T ] \leq \frac{C (\log_2 d)^2 n}{d^2}.
\end{align*}

\subsubsection{On a Lemma of \cite{indyk2003fast}}\label{section:indyk:and:thaper}

We now prove a modified result of \cite{indyk2003fast}. The main added twist is the fact that $\eta$ need not be uniform on $[0,1]^q, q \in \NN$ but can be in fact taken to be uniformly distributed on a sufficiently small grid. This has an important practical implication as it enables calculating our test statistic. Although our result contains this additional complication, the proof still follows the idea of \cite{indyk2003fast}. Let
\begin{align*}
\mathcal{Q} = \bigg\{Q^k: k \in 0,1,\ldots, \lceil \log_2(\frac{1}{\varphi})\rceil\bigg\},
\end{align*}
be a collection of grids on $[0,1]^{q}$ for $q \in \mathbb{N}$, with side (Euclidean) length $\frac{1}{2^k}$, centered at the point $\eta \in [0,1]^{q}$ (we will only use the result when $q= 1$). Here $\eta$ lies on a grid of side length $\frac{1}{2^{\lceil\log_2(\frac{1}{\varphi})\rceil + 1}}$ centered at $0$. Let $\mathbb{E}_\eta$ denote the expectation with respect to $\eta$ uniformly sampled on the aforementioned grid.
\begin{lemma}\label{indyk:thaper:on:a:grid} Then we have
\begin{align*}
\mathbb{E}_\eta \sum_{k \in [| \mathcal{Q}|]} \frac{1}{2^k} \sum_{C \in Q^k}  |\mu(C) - \nu(C)| \leq (\lceil \log_2(\frac{1}{\varphi})\rceil + 1)4 q \bigg(W_1(\mu,\nu)/\sqrt{q} +\frac{1}{2^{\lceil \log_2(\frac{1}{\varphi})\rceil + 1}}\bigg). 
\end{align*}
\end{lemma}

\begin{proof} Define
\begin{align*}
S_k = \bigg\{(x,y) : \frac{\sqrt{q}}{2^{k + 1}} < \|x - y\|_2 \leq \frac{\sqrt{q}}{2^k}\bigg\},
\end{align*}
for $k = 0,1,\ldots, \lceil \log_2(\frac{1}{\varphi})\rceil$, and let $S_{ \lceil \log_2(\frac{1}{\varphi})\rceil + 1} = \bigg\{(x,y) : \|x-y\|_2 \leq \frac{\sqrt{q}}{2^{\lceil \log_2(\frac{1}{\varphi})\rceil + 1}}\bigg\}$. Let $\gamma$ be an optimal coupling for the Wasserstein distance $W_1(\mu, \nu)$. By definition, we have the following bound
\begin{align}\label{W1:lower:bound}
W_1(\mu,\nu) \geq \sum_{\ell = 0}^{\lceil \log_2(\frac{1}{\varphi})\rceil} \frac{\sqrt{q}}{2^{\ell+1}} \int_{S_\ell} d \gamma(x,y) = \sum_{\ell = 0}^{\lceil \log_2(\frac{1}{\varphi})\rceil} \frac{\sqrt{q}}{2^{\ell+1}} \gamma(S_\ell). 
\end{align}

We will now re-express the multi-resolution $L_1$ distance in terms of $\gamma$. We have 
\begin{align}
\mathbb{E}_\eta \sum_{k \in [| \mathcal{Q}|]} \frac{1}{2^k} \sum_{C \in Q^k}  |\mu(C) - \nu(C)| & = \sum_{k \in [| \mathcal{Q}|]} \frac{1}{2^k}\mathbb{E}_\eta \sum_{C \in Q^k}   \bigg| \int_{C\times [0,1]^q} d \gamma(x,y) -  \int_{ [0,1]^q \times C} d \gamma(x,y)\bigg| \nonumber \\
& = \sum_{k \in [| \mathcal{Q}|]}\frac{1}{2^k} \mathbb{E}_\eta \sum_{C \in Q^k}   \bigg| \int_{C\times C^c} d \gamma(x,y) -  \int_{ C^c \times C} d \gamma(x,y)\bigg| \nonumber\\
& \leq \sum_{k \in [| \mathcal{Q}|]} \frac{1}{2^k} \mathbb{E}_\eta\sum_{C \in Q^k}  (\gamma(C\times C^c) + \gamma(C^c \times C))\nonumber\\
& = \sum_{k \in [| \mathcal{Q}|]} \frac{1}{2^{k-1}} \mathbb{E}_\eta\sum_{C \in Q^k}  \gamma(C\times C^c) \label{expectation:l1:identity},
\end{align}
where the last identity follows since each two distinct sets $C_1, C_2 \in Q^k$ we have $\gamma(C_1 \times C_2)$ and $\gamma(C_2 \times C_1)$ appearing once in each of the two summations. Next we will control the expression 
\begin{align*}
\mathbb{E}_\eta \sum_{C \in Q^k}   \gamma(C\times C^c) & =  \mathbb{E}_\eta \sum_{C \in Q^k} \sum_{\ell = 0}^{\lceil \log_2(\frac{1}{\varphi})\rceil + 1}  \gamma(C\times C^c \cap S_\ell) \\
& = \mathbb{E}_\eta \sum_{C \in Q^k} \sum_{\ell = 0}^{\lceil \log_2(\frac{1}{\varphi})\rceil + 1} \gamma(S_\ell)  \int \mathbbm{1}_{C \times C^c}(x,y) d \gamma(x,y| S_\ell) \\
&=   \sum_{\ell = 0}^{\lceil \log_2(\frac{1}{\varphi})\rceil + 1} \gamma(S_\ell) \int \mathbb{E}_\eta \sum_{C \in Q^k} \mathbbm{1}( (x,y) \in C \times C^c) d \gamma(x,y| S_\ell).
\end{align*}
Note that $ \mathbb{E}_\eta\sum_{C \in Q^k} \mathbbm{1}( (x,y) \in C \times C^c) = \mathbb{P}_\eta ((x,y) \in \cup_{C \in Q^k} C \times C^c)$ is the probability that the edge $(x,y) \in S_\ell$ ``crosses'' the grid $Q^k$. Let $z_1, \ldots,z_q$ be the lengths of the Euclidean projections of the vector $y -x$ on the axis. The grid is crossed if and only if any of the projections crosses a side of the grid. By the union bound this happens with probability at most 
\begin{align*}\sum_{i \in [q]} \frac{z_i + s}{\frac{1}{2^k}} & = 2^k \sum_{i \in [d]} z_i + 2^k q s \leq 2^k \sqrt{q} \sqrt{\sum_{i \in [d]} z_i^2} + 2^k q s \\
& = 2^k \sqrt{q} \|x-y\|_2 + 2^k q s \leq \frac{2^k q}{2^\ell}+ 2^k q s,
\end{align*}
where the last bound holds since $(x,y) \in S_\ell$ and $s = \frac{1}{2^{\lceil\log_2(\frac{1}{\varphi})\rceil + 1}}$ is the length of the grid for $\eta$. We conclude that 
\begin{align*}
 \sum_{C \in Q^k}  \mathbb{E}_\eta \gamma(C\times C^c) \leq  2^k q s + \sum_{\ell = 0}^{\lceil \log_2(\frac{1}{\varphi})\rceil + 1} \gamma(S_\ell)  \frac{2^k q}{2^\ell}. 
\end{align*}
Going back to \eqref{expectation:l1:identity} we have 
\begin{align*}
\mathbb{E}_\eta \sum_{k \in [| \mathcal{Q}|]} \frac{1}{2^k} \sum_{C \in Q^k}  |\mu(C) - \nu(C)| & \leq \sum_{k \in [| \mathcal{Q}|]} \frac{1}{2^{k-1}} \bigg(2^k q s +  \sum_{\ell = 0}^{\lceil \log_2(\frac{1}{\varphi})\rceil + 1} \gamma(S_\ell)  \frac{2^k q}{2^\ell} \bigg)\\
& \leq (\lceil \log_2(\frac{1}{\varphi})\rceil + 1) 2 qs \\
& + \sum_{k \in [| \mathcal{Q}|]} \sum_{\ell = 0}^{\lceil \log_2(\frac{1}{\varphi})\rceil} \gamma(S_\ell)  \frac{4d}{2^{\ell + 1}} + \sum_{k \in [| \mathcal{Q}|]} \gamma(S_{\lceil \log_2(\frac{1}{\varphi})\rceil + 1}) \frac{2 d}{2^{\lceil \log_2(\frac{1}{\varphi})\rceil + 1}}\\
& \leq (\lceil \log_2(\frac{1}{\varphi})\rceil + 1)4 q \bigg(s/2 +  \frac{W_1(\mu,\nu)}{\sqrt{q}}  +\frac{1}{2^{\lceil \log_2(\frac{1}{\varphi})\rceil + 2}}\bigg),
\end{align*}
where we used \eqref{W1:lower:bound} in the above inequality. Recalling that $s = \frac{1}{2^{\lceil\log_2(\frac{1}{\varphi})\rceil + 1}}$, the above can be made smaller than
\begin{align*}
(\lceil \log_2(\frac{1}{\varphi})\rceil + 1)4 q \bigg(\frac{W_1(\mu,\nu)}{\sqrt{q}} +\frac{1}{2^{\lceil \log_2(\frac{1}{\varphi})\rceil + 1}}\bigg),
\end{align*}
as claimed.
\end{proof}

\subsubsection{Analysis of the Variance} \label{Section: Analysis of the Variance}
We now turn to the analysis of the variance of the test statistic $T$. First of all, the rule of total variance ensures that
$$
\Var T = \EE[\Var [T | \sigma]] + \Var[\EE [T | \sigma]],
$$
where $\sigma = (\sigma_m)_{m \in [d]}$. The first term is

$$
\Var [T | \sigma] = \sum_{m,j \in [d]} \Cov(T^m, T^j | \sigma_m, \sigma_j) = \sum_{m \in [d]} \Var(T^m | \sigma_m),
$$
where $T^m = \sum_{k} 1/2^{2k} \EE_\eta U(\cD_m^k) \mathbbm{1}(\sigma_m \geq 4) \sigma_m$. Since $\Var (\sum_{i \in [k]} X_i) \leq k \sum \Var(X_i)$ (which follows by the fact that $\Cov(X,Y) \leq \Var(X)/2 + \Var(Y)/2$) we have
\begin{align*}
\Var(T^m | \sigma_m) & \leq (\log_2 d)  \sum_{k} \Var (1/2^{2k} \EE_\eta U(\cD_m^k) \mathbbm{1}(\sigma_m \geq 4) \sigma_m | \sigma_m) \\
& =  (\log_2 d)  \sum_{k}\mathbbm{1}(\sigma_m \geq 4) \sigma_m^2 \Var (1/2^{2k} \EE_\eta U(\cD_m^k)  | \sigma_m).
\end{align*}
Now, $ \Var (\EE_\eta U(\cD_m^k)  | \sigma_m )\leq \EE_\eta \Var( U(\cD_m^k)  | \sigma_m, \eta)$. This is so since
\begin{align*}
\Var \bigl[ \EE_\eta U(\cD_m^k)  | \sigma_m \bigr] & = \EE \bigl( \bigl\{ \EE_\eta  \bigl[ U(\cD_m^k) - \EE(U(\cD_m^k) | \sigma_m, \eta) \bigr] \bigr\}^2  | \sigma_m \bigr) \\
& \leq   \EE \bigl(\EE_\eta \bigl\{ \bigl[ U(\cD_m^k) - \EE(U(\cD_m^k) | \sigma_m, \eta) \bigr] \bigr\}^2  | \sigma_m \bigr) \\
& = \EE_\eta \EE  \bigl( \bigl\{  \bigl[U(\cD_m^k) - \EE(U(\cD_m^k) | \sigma_m, \eta) \bigr] \bigr\}^2  | \sigma_m, \eta \bigr) \\
& = \EE_\eta \bigl\{\Var\bigl[U(\cD_m^k) | \sigma_m, \eta\bigr]\bigr\},
\end{align*}
where we used the fact that $\eta$ is independent of all other randomness. Now from Lemma 5.1 of \cite{neykov2021minimax} we can conclude

\begin{align*}
\MoveEqLeft \sum_{m \in [d]} \Var(T^m | \sigma_m) \leq  (\log_2 d)\sum_{m \in [d]} \sigma^2_m \mathbbm{1}(\sigma_m \geq 4) C\sum_{k} \frac{1}{2^{4k}}\EE_\eta \bigg(\frac{\EE[U(\cD_m^k) | \sigma_m, \eta]}{\sigma_m} + \frac{1}{\sigma_m^2}\bigg) \\
&\leq  (\log_2 d)\sum_{m \in [d]}  C\biggl(\sum_{k} 1/2^{2k}\EE [U(\cD_m^k)\sigma_m | \sigma_m] + \mathbbm{1}(\sigma_m \geq 4)\biggr). 
\end{align*}
Taking expectation of the expression above we end up with
$$
\EE[\Var[T | \sigma]] \leq C (\log_2 d)\bigg(\EE[T] + \EE \sum_{m \in [d]} \mathbbm{1}(\sigma_m \geq 4)\bigg) \leq C (\log_2 d)(\EE[T] + d).
$$
For the second term we have
\begin{align*}
\EE[T | \sigma] & = \sum_{m \in [d]} \sigma_m \mathbbm{1}(\sigma_m \geq 4)\sum_{k} 1/2^{2k} \EE[\EE_\eta [U(\cD_m^k)] | \sigma_m] \\
& = \sum_{m \in [d]} \sigma_m \mathbbm{1}(\sigma_m \geq 4) \sum_k 1/2^{2k} \sum_{i,j} \EE_\eta (q^k_{ij}(m) - q_{i\cdot}^k(m)q_{\cdot j}^k(m))^2
\end{align*}
Since the $\sigma_m$ are independent we have
$$
\Var[\EE[T | \sigma]] = \sum_{m \in [d]} \Var[\sigma_m \mathbbm{1}(\sigma_m \geq 4)] \bigg(\sum_k 1/2^{2k} \EE_\eta \sum_{i,j} (q^k_{ij}(m) - q_{i\cdot}^k(m)q_{\cdot j}^k(m))^2\bigg)^2
$$
By Claim 2.1 of \cite{canonne2018testing}, 
we have that $ \Var[\sigma_m \mathbbm{1}(\sigma_m \geq 4)] \leq C'\EE[ \sigma_m \mathbbm{1}(\sigma_m \geq 4)]$, and $ \sum_{i,j} (q^k_{ij}(m) - q_{i\cdot}^k(m)q_{\cdot j}^k(m))^2 \leq ( \sum_{i,j} |q^k_{ij}(m) - q_{i\cdot}^k(m)q_{\cdot j}^k(m)|)^2 \leq 4$ thus
$$
\Var[\EE[T | \sigma]] \leq 4C'\sum_{m \in [d]} \EE[\sigma_m \mathbbm{1}(\sigma_m \geq 4)] \sum_{k} 1/2^{2k} \EE_\eta \sum_{i,j} (q^k_{ij}(m) - q_{i\cdot}^k(m)q_{\cdot j}^k(m))^2 = 4C' \EE[T].
$$

Hence $\Var T \leq   (\log_2 d) C(\EE[T] + d)$. 

\subsection{Putting Things Together}\label{chebyshevs:ineq:section}

Recall that the threshold $\tau = \zeta \sqrt{d}(\log_2 d)^2$ while $d \asymp n^{2/5}$. First we handle the null hypothesis. By Chebyshev's inequality we have
$$
\PP(|T - \EE T| \geq \frac{\tau}{2} ) \leq \frac{4 \Var (T )}{\tau^2} = \frac{C\log_2 d(\EE[T] + d)}{\tau^2} \leq \frac{C\log_2 d(\frac{C (\log_2 d)^2 n}{d^2} + d)}{\zeta d \log^4_2 d} \leq \frac{1}{10},
$$
when $\zeta$ is large enough. In this scenario we have that $T \leq \frac{\tau}{2} + \EE T$ which is of the order $\frac{C n}{d^2} + \sqrt{\zeta d \log^4_2 d}/2$. Under the alternative, as we argued in \eqref{Eq: lower bound on expectation}: $\EE[T] \gtrsim \min \biggl\{ \frac{n \Upsilon^2}{\lceil \log_2 d\rceil}, \, \frac{(n \Upsilon)^4}{d^3 \lceil \log_2 d\rceil^2} \bigg\}.$ When $\Upsilon \geq \frac{(\log_2 d)^{3/2}}{d}$, simple algebra (using $d \asymp n^{2/5}$) shows that $\min \biggl\{ \frac{n \Upsilon^2}{\lceil \log_2 d\rceil}, \, \frac{(n \Upsilon)^4}{d^3 \lceil \log_2 d\rceil^2} \bigg\} = \frac{n \Upsilon^2}{\lceil \log_2 d\rceil}$, so that 
$$
\PP(|T - \EE T| \geq \EE T/2) \leq \frac{4\Var T}{(\EE T)^2} \leq 4 C\bigg(\frac{d\log_2 d}{(\EE T)^2} + \frac{\log_2 d}{\EE T}\bigg) \leq \frac{1}{10},
$$
since $\EE T \geq \zeta \sqrt{d} \log^2_2 d$ in order when $d \asymp n^{2/5}$. 

\section{Lower Bound}\label{lower:bound:section}

In this section we consider a lower bound which nearly matches the upper bound from Theorem \ref{main:theorem}. The main result of this section is as follows.

\begin{theorem}[Critical Radius Lower Bound]\label{first:lower:bound} Let $L \in \RR^+$ be a fixed constant. Then for some absolute constant $c_0 > 0$ the critical radius defined in \eqref{critical:radius} is bounded as
\begin{align*}
\varepsilon_n(\cP_0^{W}(L), \cP_1^W(L,\varepsilon)) \geq  \frac{c_0}{n^{1/5}}.
\end{align*}
\end{theorem}

One can see that the minimax rate given by Theorem \ref{first:lower:bound} nearly matches (up to logarithmic factors) the rate from the upper bound of Theorem \ref{main:theorem}. Here once again we would like to stress the fact that the minimax optimal rate obtained here under $W_1$ smoothness and $W_2$ separation is slower compared to the rates obtained by \cite{neykov2021minimax}. Hence even though the separation is stronger (see Proposition \ref{W:sep:discards:more:distributions:than:TV:sep}) the added flexibility from only imposing Wasserstein smoothness drives the slower rate. The remaining of this section is dedicated to the proof of the above theorem. The techniques we use are similar to those used by \cite{neykov2021minimax}, but the worst case is quite different. Unlike in \cite{neykov2021minimax} where the authors considered cases where $p_{X,Y|Z}$ are discrete or continuous distributions and obtained different rates, here there is no need for that. The worst case is achieved by a distribution which is discrete $X,Y | Z = z$ for all $z\in [0,1]$, and in fact is concentrated only on $4$ points. Before we detail this construction we will require the following lemma, which is useful when we establish the $W_2$ separation in the alternative.

\begin{lemma}\label{Larrys:lemma}
For any distribution $p = p_{X,Y,Z}$ for $X,Y,Z \in [0,1]$,
$\psi(p)\leq \tilde \psi(p) \leq C\psi(p)$ for some absolute constant $C$
where
$$
\tilde \psi(p) =  \int W_2 (p_{XY|Z=z},p_{X|Z=z}\times p_{Y|Z=z}) d p_Z(z),
$$
and
\begin{align}\label{psi:p:def}
\psi(p) =  \inf_{q \in \cP_0} \int W_2 (p_{XY|Z=z},q_{X|Z=z}\times q_{Y|Z=z}) dp_Z(z).
\end{align}
\end{lemma}

\begin{proof}
Let $p^* = p^*_{X|Z}\times p^*_{Y|Z}$ be the minimizer of
$\psi(p)$ (if a minimizer does not exist we may take a sequence of distributions that converges to it).
Then
\begin{align*}
\tilde\psi(p) &=
\int W_2 (p_{XY|Z=z},p_{X|Z=z}\times p_{Y|Z=z}) dp_Z(z)\\
& \leq
\int W_2 (p_{XY|Z=z},p_{X|Z=z}^*\times p_{Y|Z=z}^*) dp_Z(z)\\
&\ \ \ +
\int W_2 (p_{X|Z=z}\times p_{Y|Z=z},p_{X|Z=z}^*\times p_{Y|Z=z}^*)dp_Z(z)\\
&\stackrel{\mbox{\tiny by Lemma 3 \cite{mariucci2018wasserstein}}}{\leq}
\psi(p) + 
\int \sqrt{[W_2^2 (p_{X|Z=z},p_{X|Z=z}^*) + W_2^2 (p_{Y|Z=z}, p_{Y|Z=z}^*)]} dp_Z(z)\\
&\leq
\psi(p) + 
\int [W_2 (p_{X|Z=z},p_{X|Z=z}^*) + W_2 (p_{Y|Z=z}, p_{Y|Z=z}^*)] dp_Z(z)\\
& \leq
\psi(p) +
\int W_2 (p_{XY|Z=z},p_{X|Z=z}^*\times p_{Y|Z=z}^*) dp_Z(z)+
\int W_2 (p_{XY|Z=z},p_{X|Z=z}^*\times p_{Y|Z=z}^*) dp_Z(z)\\
&= C \psi(p).
\end{align*}
In the above we used that 
\begin{align*}
\int W_2(p_{X|Z=z},p_{X|Z=z}^*) dp_Z(z)\leq \int W_2 (p_{XY|Z=z},p_{X|Z=z}^*\times p_{Y|Z=z}^*) dp_Z(z),
\end{align*}
which follows by similar arguments as in Lemma \ref{Wasserstein:smoothness:lemma}. 
\end{proof}

Let $Z$ be $U([0,1])$. For the null distributions for each $z$ we specify $q_{X,Y|Z = z} = q_{X | Z = z} q_{Y| Z = z}$ as four point masses at $(0,0), (0,1), (1,0)$ and $(1,1)$ with equal probability $\frac{1}{4}$. Under the alternative we specify $p_{X,Y|Z = z}$ as four point masses at $(0,0), (0,1), (1,0)$ and $(1,1)$, where $p_{X,Y|Z=z}(0,0) = p_{X,Y|Z=z}(1,1) = \frac{1}{4} + \delta \xi(z)$ and $p_{X,Y|Z=z}(0,1) = p_{X,Y|Z=z}(1,0) = \frac{1}{4} - \delta \xi(z)$, where $\xi(z)$ is specified as:
\begin{align*}
\xi_\nu(z) = \rho \sum_{j \in [d]} \nu_j h_{j, d}(z),
\end{align*}
where $\rho > 0$ is a constant, $\delta \in \{-1,1\}$ is a Rademacher random variable, $d \in \NN$, $\nu_i \in \{-1,+1\}$, and $h_{j, d}(z) = \sqrt{d}h(d z - j + 1)$ for $z \in [(j-1)/d, j/d]$, and $h$ is an infinitely differentiable function supported on $[0,1]$ 

such that $\int h(z) dz = 0$ and $\int h^2(z) dz = 1$. When perturbing, in order to ensure that we create valid probability distributions, we need to satisfy the conditions that 
\begin{align}\label{eqn:bounds:rho:d}
\frac{1}{4} - \rho \sqrt{d} \|h\|_{\infty} \geq 0,~~~ \frac{1}{4} + \rho \sqrt{d} \|h\|_{\infty} \leq 1.
\end{align}
Clearly, the second inequality in \eqref{eqn:bounds:rho:d} is implied by the first one and is hence redundant. We will ensure the first inequality by the choice of $\rho$ and $d$ to follow. 

We need to show that the two distributions $q$ and $p$ are Wasserstein smooth. This is obvious for $q$ which does not change with $z$. To see this for $p_{X,Y|Z = z} $ and $p_{X,Y|Z = z'} $ we will use the dual characterization of the Wasserstein distance. We have
\begin{align*}
W_1(p_{X,Y|Z = z} , p_{X,Y|Z = z'}) & = \sup_{f \in \operatorname{Lip}(1)} |f(0,0) + f(1,1) - f(0,1) - f(1,0)| |\xi(z) - \xi(z')| \\
& \leq 2 |\xi(z) - \xi(z')|\\
& \leq 2 d^{3/2}\rho \|h'\|_{\infty} |z - z'|,
\end{align*}
since the derivative of $\xi(z)$ is bounded by $d^{3/2}\rho \|h'\|_{\infty}$. Next we will handle $\psi(p)$ as defined in \eqref{psi:p:def}. We will start by checking that $W_1(p_{X,Y|Z =z }, p_{X| Z = z} p_{Y | Z = z})$ is at least $C |\xi(z)|$ for some constant $C$. Once again we use the dual characterization of the Wasserstein distance to obtain
\begin{align*}
W_1(p_{X,Y|Z =z }, p_{X| Z = z} p_{Y | Z = z}) = \sup_{f \in \operatorname{Lip}(1)} |f(0,0) + f(1,1) - f(0,1) - f(1,0)| |\xi(z)| 
\end{align*} 
Take the function $f(x,y) = \frac{1}{\sqrt{2}}|x - y|$ where $x,y \in \RR$. This is a $1$-Lipschitz function in $\|\cdot\|_2$ (since $||x-y| -|x'-y'|| \leq |x - x' + y' - y|\leq |x - x'| +| y- y'| \leq \sqrt{2((x-x')^2 + (y - y')^2)}$). It follows that 
\begin{align*}
W_1(p_{X,Y|Z =z }, p_{X| Z = z} p_{Y | Z = z}) \geq \sqrt{2} |\xi(z)|.
\end{align*}  
Next we show that $W_2^2(p_{X,Y | Z = z}, p_{X | Z = z} p_{Y | Z = z}) \asymp W_1(p_{X,Y | Z = z}, p_{X | Z = z} p_{Y | Z = z})$. This follows since when $x,y \in \{(0,0), (1,1), (0,1), (1,0)\}$ we have
\begin{align*}
\|x-y\|_2 \leq \|x - y\|^2_2 \leq \sqrt{2}\|x - y\|_2. 
\end{align*}
Hence if $\gamma_2$ is an optimal coupling for $W_2(p_{X,Y | Z = z}, p_{X | Z = z} p_{Y | Z = z})$ we have
\begin{align*}
W_2^2(p_{X,Y | Z = z}, p_{X | Z = z} p_{Y | Z = z})  & = \int \|x-y\|^2_2 d \gamma_2(x,y) \\
& \geq \int \|x-y\|_2 d \gamma_2(x,y) \geq W_1(p_{X,Y | Z = z}, p_{X | Z = z} p_{Y | Z = z}).
\end{align*}
On the other hand if $\gamma_1$ is an optimal coupling for $W_1(p_{X,Y | Z = z}, p_{X | Z = z} p_{Y | Z = z})$ we have
\begin{align}\label{W_1:geq:W_2^2}
W_2^2(p_{X,Y | Z = z}, p_{X | Z = z} p_{Y | Z = z})  & \leq \int \|x-y\|^2_2 d \gamma_1(x,y) \nonumber \\
& \leq \sqrt{2} \int \|x-y\|_2 d \gamma_1(x,y) = \sqrt{2} W_1(p_{X,Y | Z = z}, p_{X | Z = z} p_{Y | Z = z}).
\end{align}

As we argued earlier, $\sqrt{W_1(p_{X,Y | Z = z}, p_{X | Z = z} p_{Y | Z = z})} \gtrsim \sqrt{|\xi(z)|} = \sqrt{\rho |\sum_{j \in [d]} \nu_j h_{j, d}(z)|}$, where $\rho > 0$ is a constant, $d \in \NN$, $\nu_i \in \{-1,+1\}$ ,and $h_{j, d}(z) = \sqrt{d}h(d z - j + 1)$ for $z \in [(j-1)/d, j/d]$.  Now, since the functions $h_{j,d}$ have disjoint supports, it follows that $\sqrt{\rho |\sum_{j \in [d]} \nu_j h_{j, d}(z)|} = \sqrt{\rho} \sum_{j \in [d]} \sqrt{|h_{j, d}(z)|}$. Hence $\EE \sqrt{|\xi(z)|} = \EE \sqrt{\rho} \sum_{j \in [d]} \sqrt{|h_{j, d}(Z)|} = d\sqrt{\rho}\sqrt[4]{d} \frac{1}{d} = c\sqrt{\rho}\sqrt[4]{d}$, where we used the fact that $\int \sqrt{|h_{j,d}(z)|}dz = c\sqrt[4]{d} \frac{1}{d}$, for some absolute constant $c$. We conclude that 
\begin{align*}
\varepsilon = \psi(p) \gtrsim \rho^{1/2} d^{1/4}. 
\end{align*}
From here the argument can proceed precisely as in Theorem 4.1 of \cite{neykov2021minimax} where $\ell_1 = \ell_2 = 2$. We conclude that one can select $\rho \asymp \frac{1}{d^{3/2}}$ and $ \frac{1}{d} \asymp \frac{1}{n^{2/5}}$ (for some sufficiently small constants so that \eqref{eqn:bounds:rho:d} are satisfied), to yield a lower bound on $\psi(p) \gtrsim \frac{1}{\sqrt{d}} \gtrsim \frac{1}{n^{1/5}}$. This completes the proof. 

\section{Discussion}\label{discussion:section}

In this paper we considered the problem of minimax Wasserstein conditional independence testing. We proposed a novel test statistic which is nearly optimal in terms of the separation radius. Despite this, there are interesting open questions that remain to be explored. Our current theory allows only for 1-dimensional random variables $X,Y,Z$. It will be interesting (yet also very challenging) to extend our results to the multivariate setting. Furthermore, while in principle our test statistic is implementable in polynomial time, the computational cost which is bigger than linear is likely high. It would be interesting to design fast computational methods to compute our statistic, or propose a statistic which is different in nature altogether yet is minimax optimal and easily computable. Another challenging open question is whether one can come up with a calibration method for the test statistic such as the one proposed in \cite{kim2021local} which is based on local permutations. The difficulty here is the fact that Wasserstein smoothness is not strong enough to apply a result such as Lemma 1 or Lemma 2 of \cite{kim2021local}, which renders it almost impossible to argue directly that a local permutation would control the type I error. Finally, while we have addressed minimax testing, it would be interesting to study the corresponding estimation problem, possibly under $W_2$ loss function with $W_1$ smoothness. We leave these important questions for future work.

\section{Acknowledgements}
This work was partially supported by funding from the NSF grants DMS2113684 and DMS-2310632, as well as an Amazon AI and a Google Research Scholar Award to SB. MN acknowledges support from the NSF grant DMS-2113684.

\bibliographystyle{plainnat}
\bibliography{condindeptesting}

\appendix

\newpage 

\appendix

\section{Deferred Proofs}

\begin{proof}[Proof of Fact \ref{important:fact:about:W}] \leavevmode
\begin{enumerate}
\item  On bounded domains $W_2(p,q) \lesssim \TV(p,q)$ \cite[See Lemma 3 and Theorem 6.15 ][respectively]{slawski2022permuted, villani2009optimal} where $\lesssim$ denotes inequality up to absolute constant factors.
\item This result can be found in Equation~(6.4) of \cite{villani2009optimal}.
\item The proof of this result can be found on page 77 of \cite{villani2009optimal}. 
\item This follows from Lemma 3 of \cite{mariucci2018wasserstein}.
\item This follows as in \eqref{W_1:geq:W_2^2}.
\end{enumerate}
\end{proof}

\begin{proof}[Proof of Lemma \ref{weed:bach:result}] Consider first $\tilde W^2_2(p,q)=  \inf_{\gamma \in \Gamma(\mu, \nu)} \int \frac{\|x-y\|^2_2}{2} d \gamma(x,y)$. The grid $Q^k$ with side Euclidean lengths (mostly) $\frac{1}{2^k}$ centered at the point $\eta$ forms a dyadic partitioning (see Definition 1 \citep{weed2019sharp}) for the scaled norm $\|x-y\|_2/\sqrt{2}$. The only condition that we need to check is whether $\diam(Q) \leq \gamma^k$ for any $Q \in Q^k$. Since $Q \in Q^k$, then we have for $(x,y) \in Q$: $\|x-y\|_2^2\leq \frac{2}{2^{2k}}$ and hence $\frac{\|x-y\|_2^2}{2} \leq \frac{1}{2^{2k}}$, so it is a dyadic partitioning with $\gamma = \frac{1}{2}$. By Proposition 1 of \cite{weed2019sharp} we immediately conclude that 
\begin{align*}
\tilde W^2_2(p,q) \leq \frac{1}{2^{2\lceil \log_2(d) \rceil}} + \sum_{k = 1}^{\lceil \log_2(d) \rceil} \frac{1}{2^{2(k - 1)}} \sum_{A_{ij}^k \in Q^k} |p(A_{ij}^k) - q(A_{ij}^k)|.
\end{align*}
Multiplying back by $2$ on both sides proves the desired result (with a multiplicative constant $4$).  
\end{proof}

\begin{proof}[Proof of Lemma \ref{Wasserstein:smoothness:lemma}] The first inequality is true by assumption. The second inequality simply plugs in the distribution $p_{X | Z = z} p_{Y | Z =z}$ in place of $q_{X,Y|Z = z}$. We now prove the last inequality. 

We will first show that 
the distributions $p_{X,Y|Z \in C_m}$ and $p_{X,Y |Z = z'}$ for any $z' \in C_m$ are close in the $W_2$ distance. To see this, fix a $z'$ and suppose that $\gamma_z(x,y,x',y')$ is an optimal coupling between $P_{X,Y|Z}(x,y|z')$ and $P_{X,Y|Z}(x,y|z)$ which minimizes the Wasserstein distance
\begin{align*}
W^2_2(p_{X,Y| Z = z'}, p_{X,Y|Z = z}) = \int \|(x,y)- (x',y')\|_2^2 d \gamma_z(x,y,x',y'),
\end{align*}
and satisfies $\gamma_z(x,y,\infty,\infty) = P_{X,Y|Z}(x,y|z')$ and $\gamma_z(\infty,\infty, x',y') = P_{X,Y|Z}(x',y'|z)$. Such an optimal coupling exists due to Theorem 4.1 of \cite{villani2009optimal}. 
Take the mixture of such distributions over $z$, i.e., consider the coupling $\int_{C_m}\gamma_z(x,y,x',y')d \tilde P(z)$, where $d \tilde P(z) = d P(z)/\mathbb{P}(Z \in C_m)$. Note that this is a coupling between the distributions $P_{X,Y|Z}(x,y|z')$ and $P_{X,Y|Z}(x,y| z \in C_m)$, since $\int_{C_m}\gamma_z(x,y,\infty,\infty)d \tilde P(z) = P_{X,Y|Z}(x,y|z')$ and 
\begin{align*}
\int_{C_m}\gamma_z(\infty,\infty,x',y')d \tilde P(z) = \int_{C_m}P_{X,Y|Z}(x',y'|z)d \tilde P(z) = P_{X,Y|Z \in C_m}(x',y' | z \in C_m).
\end{align*}
 Now consider
\begin{align}\label{wasserstein:chain:of:inequalities}
W^2_2(p_{X,Y|Z \in C_m}, p_{X,Y|Z = z'}) &\leq \int \|(x,y)- (x',y')\|_2^2 d \int_{C_m}\gamma_z(x,y,x',y')d \tilde P(z) \nonumber \\
& = \int_{C_m} \int  \|(x,y)- (x',y')\|_2^2 d \gamma_z(x,y,x',y')d \tilde P(z)\nonumber \\
& = \int_{C_m} W_2^2(p_{X,Y| Z = z'}, p_{X,Y|Z = z}) d\tilde P(z) \nonumber\\
& \lesssim L \operatorname{diam}(C_m),
\end{align}
where the last inequality follows from the fact that $W_1 \gtrsim W_2^2$, since we are on a bounded domain and using Assumptions \ref{Wasserstein:smoothness:assumption} and \ref{assumption:W2:separation}.

We will now show that $W_1(p_{X,Y|Z = z}, p_{X,Y|Z = z'}) \geq W_1(p_{X | Z = z}, p_{X | Z=z'})$. Let $\gamma(x,y,x',y')$ denote an optimal coupling between $p_{X,Y|Z = z}$ and $p_{X,Y|Z = z'}$. Note that $\gamma(x,\infty, x', \infty) = \int_{y,y'} d \gamma(x,y,x',y')$ is a coupling between $p_{X | Z = z} $ and $p_{X | Z= z'}$. Thus we have
\begin{align}
W_1(p_{X,Y|Z = z}, p_{X,Y|Z = z'}) & = \int \|(x,y) -  (x',y')\|_2 d \gamma(x,y,x',y') \nonumber \\
& \geq \int \|(x,0) -  (x', 0)\|_2 d \gamma(x,\infty,x',\infty) \nonumber \\
& \geq  W_1(p_{X | Z = z}, p_{X | Z=z'}).\label{W_1:joint:bigger:than:marginals}
\end{align}
Similarly, one can argue that $W_1(p_{X,Y|Z = z}, p_{X,Y|Z = z'})  \geq W_1(p_{Y | Z = z}, p_{Y | Z=z'})$. Hence by the same reasoning as in \eqref{wasserstein:chain:of:inequalities} we may show that under the condition $W_1(p_{X,Y|Z = z}, p_{X,Y|Z = z'}) \leq L|z - z'|$ we have
\begin{align*}
\max\bigl(W^2_2(p_{X | Z = z}, p_{X| Z \in C_m}), W^2_2(p_{Y| Z = z}, p_{Y| Z \in C_m})\bigr) \lesssim L\operatorname{diam}(C_m). 
\end{align*}
Next we will show that for any $z \in C_m$ we have
\begin{align*}
|W_2(p_{X,Y| Z = z}, p_{X|Z = z} p_{Y| Z = z}) - W_2(p_{X,Y| Z\in C_m}, p_{X|Z \in C_m} p_{Y| Z \in C_m}) | \lesssim \bigl(L\operatorname{diam}(C_m)\bigr)^{1/2}.
\end{align*}
First we observe that:
\begin{align*}
W_2(p_{X|Z = z} p_{Y| Z = z}, p_{X|Z \in C_m} p_{Y| Z \in C_m}) & \leq \sqrt{W^2_2(p_{X|Z = z}, p_{X|Z \in C_m})  + W^2_2(p_{Y|Z = z}, p_{Y|Z \in C_m} )}
\end{align*}
where we used Lemma 3 of \cite{mariucci2018wasserstein}, which shows that the Wasserstein distance squared with a distance function equal to $\|\cdot\|_2$  norm, is sub-additive on product distributions. Next, by the triangle inequality we have
\begin{align*}
\MoveEqLeft |W_2(p_{X,Y| Z = z}, p_{X|Z = z} p_{Y| Z = z}) - W_2(p_{X,Y| Z\in C_m}, p_{X|Z \in C_m} p_{Y| Z \in C_m}) | \\
& \leq W_2(p_{X,Y| Z = z}, p_{X,Y| Z\in C_m}) + W_2(p_{X|Z = z} p_{Y| Z = z}, p_{X|Z \in C_m} p_{Y| Z \in C_m})\\
& \leq W_2(p_{X,Y| Z = z}, p_{X,Y| Z\in C_m}) + \sqrt{W^2_2(p_{X|Z = z}, p_{X|Z \in C_m})  + W^2_2(p_{Y|Z = z}, p_{Y|Z \in C_m} )}\\
& \lesssim \bigl(L\operatorname{diam}(C_m)\bigr)^{1/2}.
\end{align*}
Integrating the above inequalities over $z$ we obtain
\begin{align*}
\int_{C_m} W_2(p_{X,Y| Z = z}, p_{X|Z = z} p_{Y| Z = z})  d P(z) & \leq [W_2(p_{X,Y| Z\in C_m}, p_{X|Z \in C_m} p_{Y| Z \in C_m}) \\
& + \kappa(L\operatorname{diam}(C_m))^{1/2}]p_m,
\end{align*}
where $\kappa > 0$ is an absolute constant. Summing up over $m$ gives the inequality that we wanted to show.
\end{proof}

\end{document}